\documentclass[11pt]{amsart}
\usepackage{latexsym, amsmath, amscd, amssymb, amsthm, bm, ushort}
\usepackage[colorlinks=true,linkcolor=blue,urlcolor=blue, citecolor=blue]%
  {hyperref}
\usepackage[shortalphabetic]{amsrefs}
\usepackage[T1]{fontenc}
\usepackage{mathptmx}
\usepackage{microtype}
\usepackage{verbatim}
\usepackage[all]{xypic,xy}
\usepackage{graphicx}
\usepackage{tikz}
\usetikzlibrary{decorations.pathreplacing,calligraphy,arrows,shapes,positioning,calc,patterns}
\usepackage{fullpage}
\usepackage{algorithm}
\usepackage{algpseudocode}
\usepackage{subcaption}
\usepackage[normalem]{ulem}
\usepackage{comment} 

\usepackage[centering, includeheadfoot, hmargin=1in, vmargin=0.8in,
  headheight=30.4pt]{geometry}
\newtheorem{lemma}{Lemma}[section]
\newtheorem{theorem}[lemma]{Theorem}
\newtheorem{corollary}[lemma]{Corollary}
\newtheorem{proposition}[lemma]{Proposition}

\theoremstyle{definition}
\newtheorem{remark}[lemma]{Remark}
\newtheorem{example}[lemma]{Example}

\newtheorem{defn}[lemma]{Definition}

\theoremstyle{remark}

\newcommand{\posint}{\ensuremath{\mathbb{Z}^{+}}} 
\newcommand{\kk}{\ensuremath{\Bbbk}}

\newcommand{\FF}{\ensuremath{\mathbb{F}}}

\newcommand{\RR}{\ensuremath{\mathbb{R}}} 
 
\newcommand{\ZZ}{\ensuremath{\mathbb{Z}}} 

\renewcommand{\geq}{\geqslant}
\renewcommand{\leq}{\leqslant}

\newcommand{\conv}{\operatorname{conv}}

\newcommand{\exc}{T_{0}}

\newcommand{\polyspace}{\mathcal{L}}

\title{Toric Surface Codes and the Periodicity of Polytopes}
\author{Amelia Gibbs}
\author{Eliza Hogan}
\author{Kelly Jabbusch}
\author{Jenna Plute}
\author{Nicholas Toloczko}
\date{\today}

\begin{document}

\begin{abstract} 
    Toric codes are error-correcting codes that are derived from toric varieties, which hold a unique correspondence to lattice polytopes. In this paper, we focus on lattice polytopes $P \subseteq \RR^2$ and the toric codes they define.  
   We begin by studying \emph{period-1 polytopes} -- polytopes satisfying the property $L(tP)$ = $tL(P)$ for all $t \in \mathbb{Z}^+$, where $tP$ is the $t$-dilate of $P$, and we prove an explicit formula for the minimum distance of toric codes associated to a particular class of period-1 polytopes.
    We also apply the methods of Little and Schwarz, using Vandermonde matrices,  to compute the minimum distance of another class of period-1 polytopes.  
\end{abstract}

\maketitle

\section{Introduction}
Toric codes are the natural extensions of Reed-Solomon codes. They were originally introduced by Hansen for two-dimensional spaces, (see \cite{Hansen98} and \cite{Hansen}) and subsequently studied by various authors including Joyner \cite{Joyner04},  Little and Schenck \cite{LSchenk}, Little and Schwarz \cite{LSchwarz}, Ruano \cite{ruano}, and Soprunov and Soprunova \cite{SS1}, \cite{SS2}. To begin the construction of a toric code, we first fix a finite field $\FF_q$ and let $P \subset \mathbb{R}^m$ be a lattice polytope, which is contained in the $m$-dimensional box $[0,q-2]^m$ and has $k$ lattice points. The toric code $C_P(\FF_q)$ is obtained by evaluating linear combinations of monomials corresponding to the lattice points of $P$ over the finite field $\FF_q$.

To determine if a given code is ``good" from a coding theoretic perspective, one considers parameters associated to the code:  the block length, the dimension, and the minimum distance (see Section \ref{sec:prelim} for precise definitions). In general, an ideal code will have minimum distance and dimension large with respect to its block length. For the class of toric codes, the first two parameters are easy to compute.  Given a toric code $C_P(\FF_q)$, the block length is $N=(q-1)^m$, and the dimension of the code is given by the number of lattice points $k= |P \cap \ZZ^m|$, \cite{ruano}. The minimum distance is not as easy to compute, and various authors have given formulas for computing (or bounding) minimum distances for toric codes defined by special classes of polytopes.  In the case of toric surface codes (where $m=2$), Hansen computed the minimum distance for codes coming from Hirzebruch surfaces (\cite{Hansen98}, \cite{Hansen}), Little and Schenck determined upper and lower bounds for the minimum distance of a toric surface code by examining Minkowski sum decompositions \cite{LSchenk}, and Soprunov and Soprunova improved these bounds for surface codes by examining the Minkowski length \cite{SS1}.  In the higher dimensional case, Little and Schwarz used Vandermonde matrices to compute minimum distances of simplices and rectangular boxes \cite{LSchwarz}, and Soprunov and Soprunova extended these results in \cite{SS2}.

In this paper, we will focus on lattice polytopes $P \subset \RR^2$ and the toric codes they define.  We begin in Section \ref{Sec:MinkowskiLength} 
by investigating the Minkowski length of lattice polytopes in the plane, and how polytope operations, such as dilation and Minkowski sum, effect the Minkowski length.  In particular, we compute the Minkowski length of certain classes of polytopes.  We further investigate which polytopes are of period $1$, that is, satisfy the property that the Minkowski length of a $t$-dialate of a polytope $P$ is equal to $t$ times the Minkowski length of $P$: $L(tP) =tL(P)$.  

Building on the results of Little and Schenck \cite{LSchenk}, in Section \ref{sec:our-method}, 
we investigate toric codes which arise from two different classes of polytopes in $\RR^2$, and compute the minimum distance of the associated code.  The first class we treat are polytopes which are the Minkowski sum of three line segments, namely of the form:
$P=  m[0,{e}_{1}] + n[0,{e}_{2}] + \ell[0,{e}_{1}+{e}_{2}],$
and the second class we treat are polytopes which are the Minkoski sum of three line segments and a $2$-simplex, $\Delta$: $P=  a[0,{e}_{1}] + b[0,{e}_{2}] + c[0,{e}_{1}-{e}_{2}] + d\Delta$ (see Theorem \ref{cor:smallest-maxl-min-dist} and Theorem \ref{cor:special-quad-clipped-rect-min-dist}).

In \cite{LSchwarz}, Little and Schwarz used special configurations of points in $(\FF_q^\times)^m$ and determinants of Vandermonde matrices to compute minimum distances of codes coming from simplices and rectangular polytopes in $\RR^m$.  This process gave a quite elementary method to compute minimum distances of certain toric codes.
In Section \ref{sec:ls-method}, we use their methods to derive the minimum distance formula for toric codes arising from polytopes of the form $P = \ell \Delta + \ell[0,e_1]+\ell[0,e_2]$, where $\ell$ is any positive integer (see Theorem \ref{cor:Van-Min-Dist}).

\section{Preliminaries} \label{sec:prelim}
\subsection{Linear Codes}
We begin with some basics from coding theory.
Fix a prime power $q$,  a \emph{linear code} $C$ is a vector subspace of $\FF_{q}^{m}$ for some $m \ge 1$ where $\FF_{q}$ denotes the finite field with $q$ elements.
Any ${c} = (c_{1}, \ldots, c_{n}) \in C$ is called a \emph{codeword} of $C$ and $n$ is called the \emph{block length} of $C$.
Since $C$ is a vector space, we define the \emph{dimension} of the code $C$ to simply be the dimension of $C$ as a vector space: $k = \dim \, C$.
Intuitively, the dimension of a linear code represents the amount of information that a codeword contains.

Given two code words $a, b \in C$ the \emph{Hamming distance} between $a$ and $b$ is the number of non-zero entries in $a-b$.
The \emph{Hamming weight} of ${a} \in C$ is simply the number of non-zero entries in ${a}$.
The \emph{minimum distance} of $C$, denoted $d(C)$, is the minimum Hamming distance between pairs of distinct codewords.
Since $C$ is linear, this is equivalent to the minimum Hamming weight of any non-zero codeword.  Intuitively, the minimum distance of the code represents the number of errors the code can correct.

\subsection{Polytopes and Minkowski Sum}

Let $P$ be a subset of $\RR^m.$  We say that $P$ is a \emph{lattice polytope} if $P$ is the convex hull of its lattice (or integer) points.  Given two lattice points $a,b \in \ZZ^m$, we define the \emph{(lattice) length} of the line segment connecting them to be  the number of lattice points on the line minus $1$. 
By a \emph{primitive line segment} we mean a line segment whose only lattice points are the endpoints (and so will have length $1$).  We will denote the standard unit $m$-simplex by $\Delta \subset \RR^m$, where $\Delta$ is the $m$-dimensional lattice polytope given by the convex hull of the origin and the $m$ unit vectors in $\RR^m$. 

\begin{example} Below in Figure \ref{fig:exc-tri-simplex} are examples of polytopes in $\RR^2$: (1) the exceptional triangle, $T_0$, which is defined to be the convex hull of the lattice points $(0,0), (2,1)$ and $(1,2)$, i.e. $\exc := \conv\{(0,0),(2,1),(1,2)\}$, (2) the $2$-simplex $\Delta := \conv\{(0,0), (1,0), (0,1) \}$ and (3) $L = [0, e_1]: = \text{conv}\{(0,0), (1,0)\}$, a line segment of length $1$.  Note that each edge of $T_0$ and of $\Delta$ has length $1$.

\begin{figure}[h]
    \centering
    \begin{tikzpicture}
    [scale=1, thick, dot/.style={circle, fill=black, inner sep=2pt}]
    \node[dot] (t1) at (0,0) {};
    \node[dot] (t2) at (1,1) {};
    \node[dot] (t3) at (2,1) {};
    \node[dot] (t4) at (1,2) {};

    \draw (t1) -- (t3) -- (t4) -- (t1) -- cycle;

    \draw[anchor=north] node at (1,-.1) { $\exc$};

     \node[dot] (b1) at (4,0) {};
     \node[dot] (b2) at (5,0) {};
    \node[dot] (b3) at (4,1) {};
    \draw (b1) -- (b2) -- (b3) -- (b1) -- cycle;
        \draw[anchor=north] node at (4.5,-.1) { $\Delta$};

    \node[dot] (l1) at (7.5,0) {};
    \node[dot] (l2) at (8.5,0) {};
    \draw (l1) -- (l2);
    \draw[anchor=north] node at (8,-.1) { $L$};

    \end{tikzpicture}
    \caption{(1) Exceptional Triangle ($\exc$), (2) 2-Simplex ($\Delta$) and (3) Line Segment ($L$)}
    \label{fig:exc-tri-simplex}
\end{figure}
\end{example}

Given a polynomial in $m$ variables, we can construct a polytope in $\RR^m$ as follows. Let $f \in \kk[t_{1}^{\pm 1}, \ldots, t_{m}^{\pm 1}],$ where $\kk$ is a field.
Write \[
    f = \sum_{\alpha \in \ZZ^{m}} c_{\alpha}t^{\alpha}
\]
where $\alpha = (\alpha_{1}, \ldots, \alpha_{m})$ and $t^{\alpha} = t_{1}^{\alpha_{1}} \cdots t_{m}^{\alpha_{m}}$.
The \emph{Newton polytope} of $f$, denoted $P_f$, is defined as \[
    P_f = \conv \left\{ \alpha \in \ZZ^{m} : c_{\alpha} \neq 0 \right\} \subset \RR^{m}.
\]
Note that there are only finitely many $\alpha \in \ZZ^{m}$ with $c_{\alpha} \neq 0$.
Therefore, $P_f$ is a lattice polytope.


Given a polytope $P \subset \RR^m$ and real number $t$, we denote by $tP$, the \emph{$t$-dialate of $P$}: $tP := \{tp : p \in P\}$. 

\begin{example} \label{ex:scale}
If we consider the $2$-simplex $\Delta := \text{conv}\{(0,0), (1,0), (0,1)\}$, and a positive integer $t$, 
$t\Delta = \text{conv}\{(0,0), (t,0), (0,t)\}$.  Figure \ref{fig:scale} illustrates the case when $t=2$.

\begin{figure}[h]
\begin{tikzpicture}[scale=.8, thick, dot/.style={circle, fill=black, inner sep=2pt}]

    \node[dot] (t1) at (0,0) {};
    \node[dot] (t2) at (0,1) {};
    \node[dot] (t3) at (1,0) {};
    \draw (t1) -- (t2) -- (t3) -- (t1) -- cycle;

    \node[dot] (q1) at (4,0) {};
    \node[dot] (q2) at (5,0) {};
    \node[dot] (q3) at (6,0) {};
    \node[dot] (q4) at (4,1) {};
    \node[dot] (q5) at (4,2) {};
    \node[dot] (q6) at (5,1) {};

    \draw (q1) -- (q3)  -- (q5) -- (q1) -- cycle;

\end{tikzpicture}
\caption{$\Delta$ and $2\Delta$}
\label{fig:scale}
   \end{figure}
\end{example}

We can also perform operations on two or more polytopes.  The \emph{Minkowski sum} of two polytopes $P,Q \subset \RR^{m}$ is defined as \[
    P+Q := \left\{ p+q : p \in P, q \in Q \right\}.
\]
If a polytope is a Minkowski sum of primitive lattice segments, we'll call it a \emph{zonotope}.

\begin{example} \label{ex:Minkowski}
Figure \ref{fig:MinkowskiZono} illustrates the Minkowski sum of the primitive line segment $P:= \text{conv}\{(0,0), (2,1) \}$ with $L := [0,e_1]$.  Note that the resulting polytope is a zonotope.

\begin{figure}[h]
\begin{tikzpicture}[scale=1, thick, dot/.style={circle, fill=black, inner sep=2pt}]

    \node[dot] (t1) at (0,0) {};
    \node[dot] (t2) at (2,1) {};

    \draw (t1) -- (t2) -- (t1) -- cycle;

    \node at (2.5, 0.5) {\Large $+$};

    \node[dot] (l1) at (3,0) {};
    \node[dot] (l2) at (4,0) {};
    \draw (l1) -- (l2);

    \node at (5, 0.5) {\Large $=$};

    \node[dot] (q1) at (6,0) {};
    \node[dot] (q2) at (8,1) {};
    \node[dot] (q3) at (9,1) {};
    \node[dot] (q4) at (7,0) {};

    \draw (q1) -- (q2) -- (q3) -- (q4) -- (q1) -- cycle;

\end{tikzpicture}
\caption{Minkowski sum of $P$ and $[0,e_1]$}
\label{fig:MinkowskiZono}
   \end{figure}

\noindent Figure \ref{fig:Minkowski} illustrates the Minkowski sum of the $2$-simplex, $\Delta \subset \RR^2$, with $L := [0,e_1]$. Note that the resulting polytope is not a zonotope.

\begin{figure}[h]
\begin{tikzpicture}[scale=1, thick, dot/.style={circle, fill=black, inner sep=2pt}]

    \node[dot] (t1) at (0,0) {};
    \node[dot] (t2) at (0,1) {};
    \node[dot] (t3) at (1,0) {};
    \draw (t1) -- (t2) -- (t3) -- (t1) -- cycle;

    \node at (1.75, 0.5) {\Large $+$};

    \node[dot] (l1) at (2.5,0) {};
    \node[dot] (l2) at (3.25,0) {};
    \draw (l1) -- (l2);

    \node at (4, 0.5) {\Large $=$};

    \node[dot] (q1) at (4.75,0) {};
    \node[dot] (q2) at (4.75,1) {};
    \node[dot] (q3) at (5.5,1) {};
    \node[dot] (q4) at (6.25,0) {};
    \node[dot] (q5) at (5.5,0) {};
    
    \draw (q1) -- (q2) -- (q3) -- (q4) -- (q5) -- (q1) -- cycle;

\end{tikzpicture}
\caption{Minkowski sum of $\Delta$ and $[0,e_1]$}
\label{fig:Minkowski}
   \end{figure}
\end{example}

Being able to write a polytope, $P$, as the Minkowski sum of ``simpler'' polytopes allows for factorization of any Laurent polynomial whose Newton polytope is $P$.  In the case of Newton polytopes, note that $P_{fg} = P_{f} + P_{g}.$

\subsection{Toric Codes}
A \emph{toric code} is defined as follows: fix a lattice polytope $P \subset \RR^{m}$ and let $\polyspace_{P}$ denote the set of all $f \in \FF_{q}[t_{1}^{\pm 1}, \ldots, t_{m}^{\pm 1}]$ such that $P_f \subseteq P$ or, equivalently, \[
    \polyspace_{P} = \operatorname{span}_{\FF_{q}}\left\{ t^{\alpha} : \alpha \in P \cap \ZZ^{m} \right\}.
\]
Define the evaluation map \[ 
    \begin{array}{rcl}
        \operatorname{ev} : \polyspace_{P} & \to & \FF_{q}^{(q-1)^{m}} \\
        f & \mapsto & \left( f(\gamma) : \gamma \in (\FF_{q}^{\times})^{m} \right)
    \end{array}.
\]
The \emph{toric code associated to $P$}, denoted by $C_P$,  is defined to be the image of $\operatorname{ev}$.
It is clear that the block length of $C_{P}$ is $n=(q-1)^{m}$, and by \cite{ruano}, the dimension of $C_{P}$ is $\#P$ where $\#P := |P \cap \ZZ^{m}|$ is the number of lattice points of $P$.  
Computing the minimum distance for $C_{P}$ is a harder problem.  Note, by definition \[
    d(C_{P}) = (q-1)^{m} - \max_{0 \neq f \in \polyspace_{P}} |Z(f)|
\]
where $Z(f)$ denotes the set of all $\gamma \in (\FF_{q}^{\times})^{m}$ such that $f(\gamma)=0$.

\subsection{Minkowski Length}
Associated to an lattice polytope $P$, we have a geometric invariant, the \emph{Minkowski length} of $P$, which was first introduced by Soprunov and Soprunova \cite{SS1}.

\begin{defn}
    Let $P \subset \RR^{m}$ be a lattice polytope.
    The \emph{Minkowski length} of $P$, denoted $L(P)$, is the maximum number of non-trivial (i.e. at least one dimensional) polytopes whose Minkowski sum is contained in $P$.
    Equivalently, $L(P)$ is the maximum number of primitive line segments whose Minkowski sum is contained in $P$.
    We say a polytope contained in $P$ is a \emph{maximal decomposition} in $P$ if it can be written as the Minkowski sum of $L(P)$ non-trivial polytopes.   Given a polytope $P$, the set of its maximal decompositions is partially ordered by inclusion: if $a,b \in \RR^m$ we write
    \[a+P_1+ \cdots + P_k < b + Q_1 + \cdots + Q_l\]
    if $a+P_1+ \cdots + P_k \subsetneq b + Q_1 + \cdots + Q_l$.  A maximal decomposition is called a \emph{smallest maximal decomposition} if it is minimal with respect to this partial order. Note that a smallest maximal decomposition is a Minkowski  sum of segments.
 \end{defn}   

\begin{example}
The polytope $2\Delta \subseteq \RR^2$ in Example \ref{ex:scale} has Minkowski length $2$ and we have three different maximal decompositions in $2\Delta$, namely $Q_1:= [0,e_1] + [0,e_1], Q_2:= [0,e_2]+[0,e_2]$ and $Q_3:=[0,e_1] + [0,e_2]$.  Similarly, the polytopes resulting from the Minkowski sum in Figures \ref{fig:MinkowskiZono} and \ref{fig:Minkowski} of Example \ref{ex:Minkowski} also have Minkowski length $2$.     

As another example, consider $P:= \Delta + T \subseteq \RR^2$, where $\Delta$ is the $2$-simplex and $T:= \conv\{(0,0), (0,1), (1,1)\}$, illustrated in Figure \ref{fig:Minkowski2}.   Here we have written $P$ as the Minkowski sum of two polytopes, but we also see that $Q:= [0,e_1]+[0,e_2]+ [0,e_2] \subseteq P$, and in fact $L(P)=3.$ 

\begin{figure}[h]
\begin{tikzpicture}[scale=1, thick, dot/.style={circle, fill=black, inner sep=2pt}]

    \node[dot] (t1) at (0,0) {};
    \node[dot] (t2) at (0,1) {};
    \node[dot] (t3) at (1,0) {};
    \draw (t1) -- (t2) -- (t3) -- (t1) -- cycle;

    \node at (1.75, 0.5) {\Large $+$};

    \node[dot] (l1) at (2.5,0) {};
    \node[dot] (l2) at (2.5,1) {};
    \node[dot] (l3) at (3.5,1) {};

    \draw (l1) -- (l2) -- (l3) -- (l1) -- cycle;

    \node at (4, 0.5) {\Large $=$};

    \node[dot] (q1) at (4.75,0) {};
    \node[dot] (q2) at (4.75,1) {};
    \node[dot] (q3) at (4.75,2) {};

    \node[dot] (q4) at (5.75,2) {};
    \node[dot] (q5) at (6.75,1) {};
    \node[dot] (q6) at (5.75,0) {};
    \node[dot] (q7) at (5.75,1) {};

    \draw (q1) -- (q2) -- (q3) -- (q4) -- (q5) -- (q6)-- (q1) -- cycle;

\end{tikzpicture}
\caption{Minkowski sum of $\Delta$ and $T$}
\label{fig:Minkowski2}
   \end{figure}
\end{example}

Newton polytopes allow us to interpret $L(P)$ in $\polyspace_{P}$, namely, $L(P)$ is the maximum number of irreducible polynomials in a factorization of any polynomial in $\polyspace_{P}$.
To see this, take $f \in \polyspace_{P}$ and write $f=f_{1} \cdots f_{r}$ where each $f_{i}$ is a nonconstant irreducible polynomial, then $P_f = P_{f_{1}} + \cdots + P_{f_{r}} \subseteq P$ so $r \le L(P)$.
Conversely, suppose that $Q_{1} + \cdots + Q_{L(P)} = Q \subseteq P$ is a maximal decomposition then $f = f_{1} \cdots f_{L(P)} \in \polyspace_{P}$ where $f_{i} = \sum_{\alpha \in Q_i \cap \ZZ^{m}} t^{\alpha}$ since $P_f = Q$.

The Minkowski length is a geometric invariant in the sense that for any affine transformation $T$ of the form $T({x}) = A{x} + \lambda$ where ${\lambda} \in \ZZ^{m}$ and $A \in \operatorname{GL}(m,\ZZ)$, we have $L(P) = L(T(P))$.
We denote the set of all such affine transformations by $\operatorname{AGL}(m, \ZZ)$.
To see this, first take a maximal decomposition in $P$, say $Q = Q_{1} + \cdots + Q_{L(P)}$, then $T(Q) = \mathbf{\lambda} + AQ_{1} + \cdots +AQ_{L(P)} \subseteq T(P)$ so $L(P) \le L(T(P))$.
Note that $T^{-1}({x}) = A^{-1}{x} - A^{-1}{\lambda}$ is contained in $\operatorname{AGL}(m,\ZZ)$.
Furthermore, $(T \circ T^{-1})({x}) = {x} = (T^{-1} \circ T)({x})$ so, by the above argument, $L(T(P)) \le L(T^{-1}(T(P))) = L(P)$.
Therefore, we have that $L(P) = L(T(P))$.
This motivates a notation of equivalence between lattice polytopes.
\begin{defn}
    Let $P,Q \subset \RR^{m}$ be lattice polytopes.
    $P$ and $Q$ are \emph{lattice equivalent}, denoted by $P \approx Q$, if there exists a $T \in \operatorname{AGL}(m,\ZZ)$ such that $T(P) = Q$.
\end{defn}

We note that if two polytopes are lattice equivalent, they have the same number of lattice points, the same number of boundary lattice points and the same Minkowski length.  Furthermore, lattice equivalence is useful when discussing toric codes due to a result of Little and Schwarz \cite[Theorem 4]{LSchwarz} which states that lattice equivalent polytopes produce monomially equivalent codes.
If $C_{1}$ and $C_{2}$ are both codes with block length $n$ and dimension $k$ over $\FF_{q}$ and $G_{1}$ is a generator matrix for $C_{1}$, then $C_{1}$ and $C_{2}$ are said to be \emph{monomially equivalent} if there exists a $n \times n$ invertible diagonal matrix $D$ and a $n \times n$ permutation matrix $\Pi$ so that \[
    G_{2} = G_{1}D\Pi
\]
is a generator matrix for $C_{2}$.
Note that monomially equivalent codes have the same block length, dimension, and minimum distance.  In particular, if $P \approx Q$, then $d(C_P) = d(C_Q).$

\section{Minkowski Length and Polytope Operations}\label{Sec:MinkowskiLength}

There are algorithms for computing the Minkowski length of a polytope in $\RR^{2}$ \cite{SS1} and in $\RR^{3}$ \cite{Beckwith_2012}, both of which run in polynomial time in $\#P$, the number of lattice points in $P$. Soprunov and Soprunova \cite{SS3} proved that the Minkowski length of $tP$ is quasi-linear when $t$ is sufficiently large for fixed $P$ and computed the Minkowski length of certain classes of polytopes, which we record below:
\begin{proposition}\label{prop:SSMinkResults}
Let $\Delta \subset \RR^{m}$ denote the standard $m$-simplex, $\Pi = \alpha_{1}[0, {e}_{1}] + \cdots + \alpha_{d}[0, e_{d}] \subset \RR^{m}$, where $e_i \subset \RR^m$ are the standard vectors, and $T \subseteq \RR^2$ a triangle in the plane.  Then

\begin{enumerate} 

    \item \cite[Theorem 2.1]{SS3} $L(t\Delta)=t$ 
    \item \cite[Example 2.3]{SS3}  $L(\Pi) = \alpha_{1} + \cdots + \alpha_{m}$.
    \item \cite[Corollary 3.2]{SS3}  $L(T) = \ell(T)$ where $\ell(T)$ denotes the \emph{lattice diameter} of $T$.
\end{enumerate}

\end{proposition}

This still leaves large families of ``simple'' polytopes without explicit formulas for their Minkowski lengths; we seek to fill some of this gap.  To begin, we note a few basic properties of the Minkowski length.
\begin{proposition}\cite[Proposition 1.2]{SS1} \label{prop:SS1MinkLength}
    Let $P,P_1, P_2, Q \subset \RR^{m}$ be lattice polytopes.
    \begin{enumerate}
        \item $L(P)$ is $\operatorname{AGL}(m,\ZZ)$-invariant.
                \item $L(P) \ge 1$ if and only if $\dim P > 0$.
        \item If $P_1+P_2 \subseteq P$, then $L(P_1) + L(P_2) \le L(P)$.
        \item If $Q$ is maximal for $P$, then $Q$ contains a zonotope $Z$ maximal for $P$.  
    \end{enumerate} 
\end{proposition}
Furthermore, Soprunov and Soprunova offer a complete classification of polytopes in $\RR^{2}$ with Minkowski length equal to 1, equivalently all polytopes which can be summands in a maximal decomposition (up to lattice equivalence).
\begin{theorem}\cite[Theorem 1.4]{SS1}
    If $P \subset \RR^{2}$ is a lattice polytope with $L(P)=1$ then $P$ is a primitive line segment, $P \approx \Delta$, or $P \approx \exc$, where $\exc = \conv \{ (0,0),(1,2), (2,1)\}$ denotes the exceptional triangle (see Figure \ref{fig:exc-tri-simplex}).   
\end{theorem}

With these foundational facts, we begin to investigate Minkowski length and how polytope operations affect the Minkowski length.

\subsection{Polytopes in $\RR^2$}
A natural place to begin is with the class of polytopes in $\RR^{2}$ that contain all polytopes which are lattice equivalent to smallest maximal decompositions.
In \cite[Proposition 3.1]{SS1}, it was shown that any smallest maximal decomposition is lattice equivalent to \[
    Q = m[0,{e}_{1}] + n[0,{e}_{2}] + \ell[0,{e}_{1}+{e}_{2}]
\]
for some integers $m,n,\ell \ge 0$.
When $Q$ is a maximal decomposition in some other polytope, then $L(Q)=m+n+\ell$ by definition.
However, we do not know that $Q$ is a maximal decomposition in some other polytope for arbitrary $m,n,\ell$.
To begin, we prove a lemma relating the number of lattice points on the boundary of a polytope to those of its subpolytopes.
To establish this relation, we construct an injective map from the lattice points which lie on the boundary of a subpolytope to those of the polytope containing it. Given a polytope $P$, $\partial P$ denotes the set of lattice points lying on the boundary of $P$.

\begin{lemma} \label{lemma:boundary-pt-ineq}
    Let $P \subset \RR^2$ be a lattice polytope and $Q \subseteq P$ a subpolytope. 
    Suppose that $P = m[0, {e}_1] + n[0, {e}_2] + \ell[0, {e}_1 + {e}_2] + s[0, {e}_{1} - {e}_{2}]$, then  $\#\partial Q \le \#\partial P$.  Moreover $\#\partial Q = \#\partial P$ if and only if $Q=P.$
\end{lemma}
\begin{proof}
Let $P = m[0, {e}_1] + n[0, {e}_2] + \ell[0, {e}_1 + {e}_2] + s[0, {e}_{1} - {e}_{2}]$, and let $Q \subseteq P$ be a subpolytope.  We will construct an injective map $f: \partial Q \to \partial P$. We begin by partitioning the lattice points on the boundary of a polytope into three sets: 

\begin{eqnarray*} 
A_Q&=& \{ (a,b) \in \partial Q \,: \,(a,b) \text{ lies on a vertical edge of $Q$} \},\\
B_Q&=&\{ (a,b) \in \partial Q \setminus A_Q \,:\, (a,b) \text{ lies on an edge of $Q$, where the outward normal has a positive $y$-component} \},\\
C_Q&=&\{ (a,b) \in \partial Q \setminus (A_Q\cup B_Q) \,:\, (a,b) \text{ lies on an edge of $Q$, where the outward normal has a negative $y$-component} \}.
\end{eqnarray*}

In Figure \ref{fig:PandQ} below, the points in $A_Q$ are red, the points in $B_Q$ are in green, and the points in $C_Q$ are in blue.

\begin{figure}[h]
\begin{tikzpicture}[scale=0.6]
    \tikzset{
        p_line/.style={line width=1.8pt, black},
        q_line/.style={line width=1.2pt, black},
        p_node/.style={circle, fill=black, inner sep=2.5pt},
        q_red/.style={circle, fill=red, inner sep=2.5pt},
        q_blue/.style={circle, fill=blue, inner sep=2.5pt},
        q_green/.style={circle, fill=green!50!black, inner sep=2.5pt},
        dot/.style={circle, fill=black, inner sep=1pt}
    }

    \draw[p_line] (0,0) -- (6,0) -- (9,3) -- (9,5) -- (3,5) -- (0,2) -- cycle;

    \draw[q_line] (2,1) -- (5,1) -- (5,4) -- (3,3) -- (2,2) -- cycle;

    \foreach \p in {(0,0), (1,0), (2,0), (3,0), (4,0), (5,0), (6,0), 
                    (7,1), (8,2), (9,3), (9,4), (9,5), 
                    (8,5), (7,5), (6,5), (5,5), (4,5), (3,5), 
                    (2,4), (1,3), (0,2), (0,1)} {
        \node[p_node] at \p {};
    }

    \node[q_red] at (2,1) {};
    \node[q_blue] at (3,1) {};
    \node[q_blue] at (4,1) {};
    \node[q_red] at (5,1) {};
    \node[q_red] at (5,2) {};
    \node[q_red] at (5,3) {};
    \node[q_red] at (5,4) {};
    \node[q_green] at (3,3) {};
    \node[q_red] at (2,2) {};

    \foreach \p in {(1,1), (1,2), (2,3), (3,2), (3,4), (4,2), (4,3), (4,4), 
                    (6,1), (6,2), (6,3), (6,4), (7,2), (7,3), (7,4), (8,3), (8,4)} {
        \node[dot] at \p {};
    }

    \node at (1, 4.5) {\fontsize{14}{14}\selectfont $P$};
    \node at (3.5, 2.3) {\fontsize{14}{14}\selectfont $Q$};

\end{tikzpicture}
 \caption{Partition of $\partial Q$ into $A_Q$ (red), $B_Q$ (green) and $C_Q$ (blue)}
  \label{fig:PandQ}
\end{figure}

Note that each lattice point $(a,b)\in P$ lies on a vertical line and a horizontal line: let $M$ denote the maximal $y$-coordinate of the vertical line $x=a$ contained in $P$, and $m$ the minimal $y$-coordinate, so that $(a,y) \in P$ for $m \leq y \leq M$, and $(a,m), (a,M) \in \partial P$. Similarly, let $N$ denote the maximal $x$
-coordinate of the horizontal line $y=b$ contained in $P$, and $n$ the minimal $x$-coordinate.  We then define $f:\partial Q \to \partial P$ as follows:  
\[f(a,b) = \begin{cases} 
(a,M) & \text{ if } (a,b)\in B_Q \\
(a,m) & \text{ if } (a,b)\in C_Q \\
(n,b) & \text { if } (a,b) \in A_Q \text{ and $(a,b)$ is on a vertical edge with outer unit normal $\langle -1,0 \rangle$} \\
(N,b) & \text { if } (a,b) \in A_Q \text{ and $(a,b)$ is on a vertical edge with outer unit normal $\langle 1,0 \rangle$,}
\end{cases}\]
see Figure \ref{fig:mapQtoP}.  Then by construction $f$ is injective, and so $\# \partial Q \leq \# \partial P.$

\begin{figure}[h]

\begin{tikzpicture}[scale=0.6]
    \tikzset{
        p_line/.style={line width=1.8pt, black},
        q_line/.style={line width=1.2pt, black},
        p_node/.style={circle, fill=black, inner sep=2.5pt},
        q_red/.style={circle, fill=red, inner sep=2.5pt},
        q_blue/.style={circle, fill=blue, inner sep=2.5pt},
        q_green/.style={circle, fill=green!50!black, inner sep=2.5pt},
        dot/.style={circle, fill=black, inner sep=1pt}
    }

    \draw[p_line] (0,0) -- (6,0) -- (9,3) -- (9,5) -- (3,5) -- (0,2) -- cycle;

    \draw[q_line] (2,1) -- (5,1) -- (5,4) -- (3,3) -- (2,2) -- cycle;

    \foreach \p in {(0,0), (1,0), (2,0), (3,0), (4,0), (5,0), (6,0), 
                    (7,1), (8,2), (9,3), (9,4), (9,5), 
                    (8,5), (7,5), (6,5), (5,5), (4,5), (3,5), 
                    (2,4), (1,3), (0,2), (0,1)} {
        \node[p_node] at \p {};
    }

    \node[q_red] at (2,1) {};
    \node[q_blue] at (3,1) {};
    \node[q_blue] at (4,1) {};
    \node[q_red] at (5,1) {};
    \node[q_red] at (5,2) {};
    \node[q_red] at (5,3) {};
    \node[q_red] at (5,4) {};
    \node[q_green] at (3,3) {};
    \node[q_red] at (2,2) {};

    \foreach \p in {(1,1), (1,2), (2,3), (3,2), (3,4), (4,2), (4,3), (4,4), 
                    (6,1), (6,2), (6,3), (6,4), (7,2), (7,3), (7,4), (8,3), (8,4)} {
        \node[dot] at \p {};
    }

    \node at (1, 4.5) {\fontsize{14}{14}\selectfont $P$};
    \node at (3.5, 2.3) {\fontsize{14}{14}\selectfont $Q$};

    \draw[teal, ultra thick, ->] (3, 3) -- (3, 4.8);
    \fill[teal] (2, 3) circle (3pt);

    \draw[red, ultra thick, <-] (0.2, 2) -- (1.8, 2);
    \draw[red, ultra thick, <-] (0.2, 1) -- (1.8, 1);
    \draw[red, ultra thick, ->] (5, 4) -- (8.8, 4);
    \draw[red, ultra thick, ->] (5, 3) -- (8.8, 3);
    \draw[red, ultra thick, ->] (5, 2) -- (7.8, 2);
    \draw[red, ultra thick, ->] (5, 1) -- (6.8, 1);

    \draw[blue, ultra thick, ->] (3, 0.8) -- (3, 0.2);
    \draw[blue, ultra thick, ->] (4, 0.8) -- (4, 0.2);

\end{tikzpicture}
\caption{$f: \partial Q \to \partial P$}
  \label{fig:mapQtoP}
\end{figure}

To prove the second part of the claim, we must show that if $\#\partial Q = \# \partial P$, then $Q=P$.  Suppose $Q$ and $P$ have the same number of lattice points in the boundary, so that the function $f$ is bijective.  We first claim that the vertical edges of $Q$ must coincide with vertical edges of $P$, i.e. $A_Q=A_P$.  If not, then we will have a vertical edge of $Q$, where by extending the edge to a vertical line, we reach a point on the boundary of $P$ which has no preimage under the map of $f$, which is a contradiction (see Figure \ref{fig:boundaryequal}, where the point on the boundary of $P$ represented by a star has no preimage under $f$).

\begin{figure}[h]
\begin{tikzpicture}[scale=0.6]
    \tikzset{
        p_line/.style={line width=1.8pt, black},
        q_line/.style={line width=1.2pt, black},
        p_node/.style={circle, fill=black, inner sep=2.5pt},
        q_red/.style={circle, fill=red, inner sep=2.5pt},
        q_blue/.style={circle, fill=blue, inner sep=2.5pt},
        q_green/.style={circle, fill=green!50!black, inner sep=2.5pt},
        dot/.style={circle, fill=black, inner sep=1pt}
    }

    \draw[p_line] (0,0) -- (4,0) -- (6,2) -- (6,5) -- (2,5) -- (0,3) -- cycle;

    \draw[q_line] (1,0) -- (4,0) -- (4,4) -- (1,4)-- cycle;

    \foreach \p in {(0,0), (0,1),(0,2),(0,3), (2,5), (3,5), (5,5),(6,5), (6,4), (6,3), (6,2), (5,1)} {
        \node[p_node] at \p {};
    }
    
 \node[star, star points=5, star point ratio = 0.5, fill=black, inner sep=4.0pt] () at (4,5) {};

    \node[q_red] at (1,0) {};
        \node[q_red] at (1,1) {};
    \node[q_red] at (1,2) {};
    \node[q_red] at (1,3) {};
    \node[q_red] at (1,4) {};
    \node[q_red] at (4,0) {};
    \node[q_red] at (4,1) {};
    \node[q_red] at (4,2) {};
    \node[q_red] at (4,3) {};
    \node[q_red] at (4,4) {};

    \node[q_blue] at (2,0) {};
    \node[q_blue] at (3,0) {};
    
    \node[q_green] at (2,4) {};
    \node[q_green] at (3,4) {};

    \foreach \p in {(2,1), (3,1), (2,2),(2,3),(3,2), (3,3), (5,2), (5,3),(5,4)} {
        \node[dot] at \p {};
    }

    \node at (.5, 5) {\fontsize{14}{14}\selectfont $P$};
    \node at (2.5, 2.3) {\fontsize{14}{14}\selectfont $Q$};

    \draw[teal, ultra thick, ->] (2, 4) -- (2, 4.8);
       \draw[teal, ultra thick, ->] (3, 4) -- (3, 4.8);

    \draw[red, ultra thick, <-] (.2,0) -- (1, 0);
    \draw[red, ultra thick, <-] (0.2, 1) -- (1, 1);
     \draw[red, ultra thick, <-] (0.2, 2) -- (1, 2);
    \draw[red, ultra thick, <-] (0.2, 3) -- (1, 3);

    \draw[red, ultra thick, ->] (4,1) -- (4.8, 1);
    \draw[red, ultra thick, ->] (4,2) -- (5.8, 2);
    \draw[red, ultra thick, ->] (4,3) -- (5.8, 3);
    \draw[red, ultra thick, ->] (4,4) -- (5.8, 4);

\end{tikzpicture}
\caption{The boundary point on $P$ represented by a star has no preimage under $f$}
  \label{fig:boundaryequal}
\end{figure}
 If we instead consider the horizontal edges of $Q$, we see that they must also coincide with the horizontal edges of $P$. Thus, we conclude that $Q=P.$ \end{proof}

\begin{remark}  One can show a similar statement for a subpolytope $Q \subseteq P = m[0, {e}_1] + n[0, {e}_2] + \ell[0, {e}_1 + {e}_2] + s[0, {e}_{1} - {e}_{2}] + r\Delta$.  However the argument is more technical, and since it will not be needed in the future, we omit the proof.  
\end{remark}

With Lemma \ref{lemma:boundary-pt-ineq}, we can deduce the Minkowski length of more polytopes than those we were initially interested in.

\begin{proposition} \label{prop:quad-clipped-rect-minkowski}
    Let $P \subset \RR^2$ be the lattice polytope \[
       P = m[0, {e}_1] + n[0, {e}_2] + \ell[0, {e}_1 + {e}_2] + s[0, {e}_{1} - {e}_{2}] + r\Delta,
    \]
    for some integers $m,n,\ell, s, r \geq 0.$
    Then, \[
        L(P) = m+n+s+\ell+\begin{cases}
            \left\lfloor \frac{3r}{2} \right\rfloor & \text{if $r < 2\ell$} \\
            \ell+r & \text{if $r \ge  2\ell$}
        \end{cases}.
    \]
\end{proposition}
\begin{proof}
    \begin{enumerate}
        \item[(1)] Suppose that $r < 2\ell$ (see Figure \ref{fig:quad-clipped-rect-l}).
        Note that \[ Q_0 = (m+r)[0, {e}_1] + (n+r)[0, {e}_2] + \left\lfloor \ell - \frac{r}{2} \right\rfloor[0, {e}_1+{e}_2] + s[0,{e}_{1}-{e}_{2}] \subseteq P \] and $L(Q_0) \ge m+n+s+2r+\left\lfloor \ell - \frac{r}{2} \right\rfloor = m+n+s+\ell+\left\lfloor \frac{3r}{2} \right\rfloor$ which gives us a lower bound on $L(P)$.
        Let $Q \subseteq P$ be a smallest maximal decomposition.  By \cite[Proposition 3.1]{SS1},  $Q = {a} + n_1E_1 + n_2E_2 + n_3E_3 \subseteq P$ where $n_i \ge 0$, each $E_i$ is a primitive line segment, and $n_1+n_2+n_3 = L(P)$.

        First, suppose that $\dim Q = 2$.
        Since $P$ satisfies the conditions of Lemma \ref{lemma:boundary-pt-ineq}, we have that $\#\partial P \ge \#\partial Q$.
        We know that \begin{align*}
            \#\partial P &= 3r+2(m+n+s+\ell) \\
            \#\partial Q &= 2(n_1+n_2+n_3) = 2L(P)
        \end{align*}
        which gives us that $m+n+\ell+s+3r/2 \ge L(P)$.
       Since $L(P)$ is an integer, we further have that \[
            L(P) \le \left\lfloor m+n+s+\ell+\frac{3r}{2} \right\rfloor = m+n+s+\ell+\left\lfloor \frac{3r}{2} \right\rfloor. 
        \]

        Now, supposing that $\dim Q = 1$ and without loss of generality, we may write $Q = {b} + n_{1}[0, {v}]$ where ${b} \in Z$ has minimal $y$-coordinate.
        By choice of ${b}$, the $y$-coordinate of ${v}$ must be non-negative.
        If the $y$-coordinate of ${v}$ is $0$ then $n_{1} \le m+s+\ell+r \le m+n+s+\ell+\lfloor 3r/2 \rfloor$ as desired.
        Alternatively, if the $y$-coordinate of ${v}$ is at least 1 then $n_{1} \le n+s+\ell+r \le m+n+s+\ell+\lfloor 3r/2 \rfloor$ as desired.
        The upper bound established in both cases and the previously established lower bound gives us \[ 
            L(P) = m+n+s+\ell+\left\lfloor \frac{3r}{2} \right\rfloor.
        \]

            \begin{figure}[h]
        \centering
        \begin{subfigure}{0.4\textwidth}
            \includegraphics[width=\textwidth]{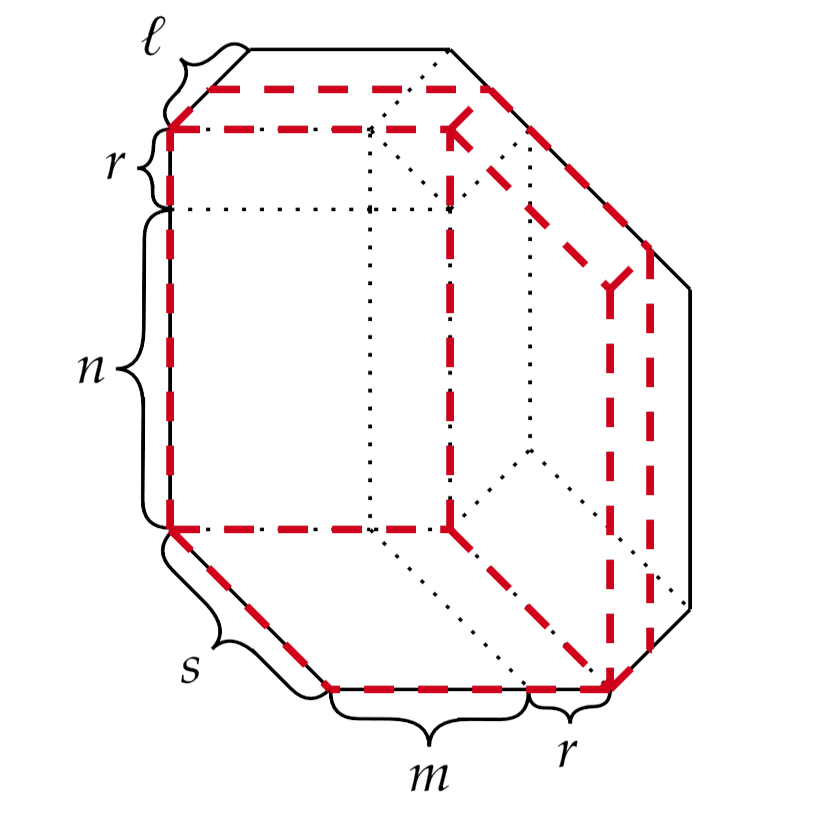}
            \caption{$r < 2\ell$}
            \label{fig:quad-clipped-rect-l}
        \end{subfigure}
        \begin{subfigure}{0.4\textwidth}
            \includegraphics[width=\textwidth]{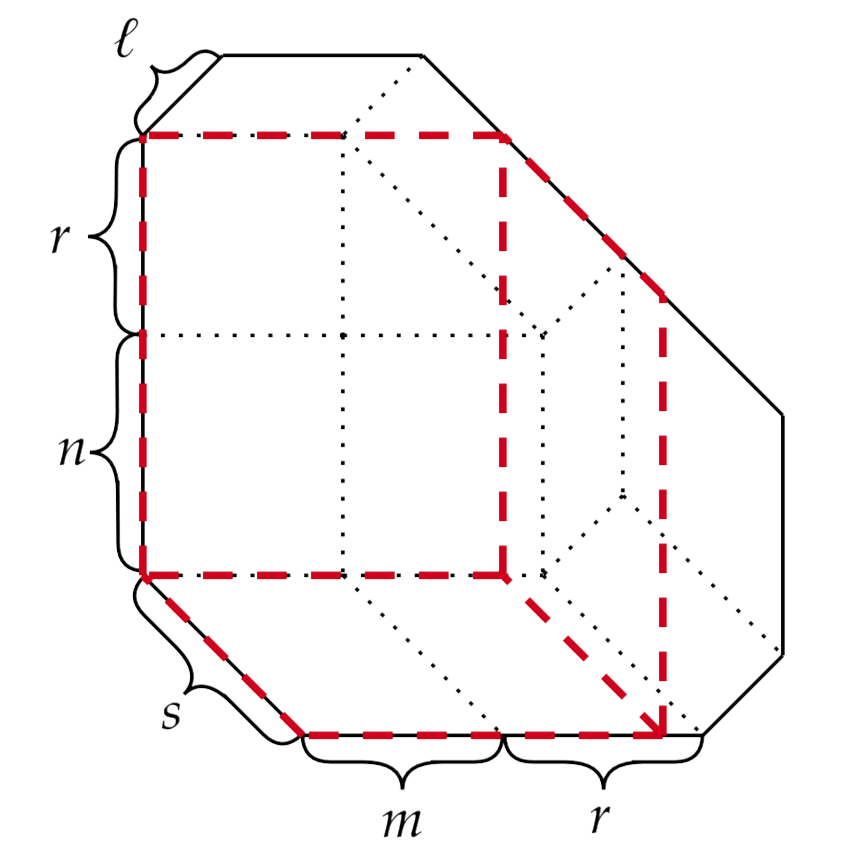}
            \caption{$r \ge 2\ell$}
            \label{fig:quad-clipped-rect-ge}
        \end{subfigure}
        \caption{$P$ (black) and maximal decompositions (red) described in Proposition \ref{prop:quad-clipped-rect-minkowski}}
    \end{figure}
        
        \item[(2)] Suppose that $r \ge 2\ell$ (see Figure \ref{fig:quad-clipped-rect-ge}). 
        We first establish that the zonotope \[ Z = (m+2\ell)[0,{e_1}] + (n+r)[0, {e_2}] \subseteq m[0,{e}_{1}] + n[0,{e}_{2}] + \ell[0,{e}_{1}+{e}_{2}] + r\Delta = P'. \]
        To do this, we just need to check that the points $(0,0), (0,n+r), (m+2\ell, 0),$ and $(m+2\ell, n+r)$ are in $P'$.
        By definition of the Minkowski sum, we know that $(0,0), (0,n+r), (m+2\ell, 0) \in P'$. 
       The point $(m+2\ell, n+r)$ is on the line $x+y=m+n+r+2\ell$ connecting $(m+r+\ell, n+\ell) \in Q'$ to $ (m+\ell, n+r+\ell) \in P'$, so by convexity of $P'$, $(m+2\ell, n+r) \in P'$.
        Note that $Z \subseteq P'$ implies $Q = Z'+s[0,{e}_{1}-{e}_{2}] \subseteq P' + s[0,{e}_{1}-{e}_{2}] = P$.
        Next, we show that $(0,s) + P \subseteq (m+n+s+r+2\ell)\Delta$. 
        Every point $(x,y) \in (0,s)+P$ must satisfy $x, y \geq 0$ and $x+y \leq m+n+s+r+2\ell$ among other conditions. 
        These bounds describe $(m+n+s+r+2\ell)\Delta$ exactly, so $(0,s) + P \subseteq (m+n+s+r+2\ell)\Delta$. 
        By \ref{prop:SSMinkResults}(a), the Minkowski length of this simplex will be $m+n+s+r+2\ell$. 
        From these inclusions, we have \[ m+n+s+r+2\ell = L(Q) \leq L(P) \leq L\left( (m+n+s+r+2\ell)\Delta \right) = m+n+s+r+2\ell. \]
        Therefore, $L(P)=m+n+s+r+2\ell$.
    \end{enumerate}\end{proof}

\begin{corollary} \label{cor:smallest-zonotope-minkowski}
    Let $P \subset \RR^2$ be the lattice polytope\[
      P = m[0, {e}_1] + n[0, {e}_2] + \ell[0, {e}_1 + {e}_2]. 
   \]
   Then $L(P) = m+n+\ell$.
\end{corollary}
\begin{proof}
   Note that $P$ is described by the polytope from the preceding proposition when $s = r = 0$. 
   In the case where $\ell > 0$, we have $r < 2\ell$. 
   Applying Proposition \ref{prop:quad-clipped-rect-minkowski} gives $L(P) = m + n + \ell$. 
   When $\ell = 0$, we have $r \geq 2\ell$ which implies $L(P) = m + n$. 
   Since $\ell=0$, we may write $L(P) = m + n + \ell$. 
\end{proof}

\subsection{Periodicity}
In the preceding  corollary,  the polytopes satisfy the property that $L(tP) = tL(P)$ for all $t \in \posint$.
However, this is not true for all polytopes --- for example the polytope in Proposition \ref{prop:quad-clipped-rect-minkowski} when $r < 2\ell$ and $2 \nmid r$, does not satisfy this condition.
For another example, consider the exceptional triangle $\exc = \conv\{(0,0),(1,2),(2,1)\}$. 
We have that $L(\exc) = 1$ and\begin{equation} \label{eq:scaled-exc-minkowski}
    L(t\exc) = t + \left\lfloor \frac{t}{2} \right\rfloor
\end{equation}
by \cite[Example 4.1]{SS3}.
In general, we can only say that \[
    L(tP) \ge tL(P).
\]

\begin{defn}
    Let $P \subset \RR^{m}$ be a lattice polytope.
    We say that $P$ has \emph{period 1} if $L(tP)=tL(P)$ for all $t \in \posint$.
    We say that $P$ has \emph{period strictly greater than 1} if $P$ is not a period 1 polytope.
\end{defn}
We note that this definition is equivalent to that given \cite[Definition 2.17]{SS3}. 
If $P$ has period 1 (in the sense of \cite[Definition 2.17]{SS3}) then \[
    \frac{L(1 \cdot P)}{1} = \sup_{t \in \posint} \frac{L(tP)}{t} \iff tL(P) \ge L(tP)
\]
for all $t \in \posint$.
So we have $tL(P)=L(tP)$ for all $t \in \posint$, and the two definitions are equivalent.

As we have seen, line segments are period $1$ polytopes and simplices are period $1$ polytopes which leaves the exceptional triangle as the only polytope which can appear in a maximal decomposition that has period greater than 1.
As such, we begin by investigating the periodicity of polytopes which contain the exceptional triangle as a summand in some maximal decomposition. 
\begin{proposition} \label{prop:period-1-nicity}
    Let $P \subset \RR^{2}$ be a lattice polytope which has a maximal decomposition $Q \subseteq P$ containing the exceptional triangle as a summand, then $P$ has period strictly greater than 1. 
\end{proposition}
\begin{proof}
    We may write \[
        Q \approx \exc + m[0,{e}_{1}] + n[0,{e}_{2}] + \ell[0,{e}_{1}+{e}_{2}]
    \]
    by \cite[Theorem 1.6]{SS1}.
    Using Equation (\ref{eq:scaled-exc-minkowski}), we have \[
        L(tP) = L(tQ) \ge L(t\exc) + L(t(m[0,{e}_{1}] + n[0,{e}_{2}] + \ell[0,{e}_{1}+{e}_{2}])) = t(m+n+\ell+1) + \left\lfloor \frac{t}{2} \right\rfloor > tL(Q)
    \]
    when $t > 1$.
\end{proof}

\section{Using maximal decompositions to compute minimum distances of toric surface codes} \label{sec:our-method}

In this section we will investigate toric codes which arise from two different classes of period-$1$ polytopes in $\RR^2$.  The first class we treat are polytopes which are the Minkowski sum of three line segments, namely of the form:
\[
  P=  m[0,{e}_{1}] + n[0,{e}_{2}] + \ell[0,{e}_{1}+{e}_{2}],
\]
and the second class we treat are polytopes which are the Minkoski sum of three line segments and a $2$-simplex:
\[
  P=  a[0,{e}_{1}] + b[0,{e}_{2}] + c[0,{e}_{1}-{e}_{2}] + d\Delta.
\]

\subsection{Polytopes which are the Minkowski sum of three line segments}
We will examine lattice polytopes in $\RR^{2}$ of the form \[
  P=  m[0,{e}_{1}] + n[0,{e}_{2}] + \ell[0,{e}_{1}+{e}_{2}].
\]
We begin by classifying all maximal decompositions in $P$; and, in doing so, prove that all polytopes of the above form are smallest maximal decompositions in $\RR^{2}$.

\begin{lemma} \label{lemma:max-decomps-in-smallest-maxl}
    Let $P \subset \RR^2$ be the lattice polytope  \[
        P = m[0,{e}_1] + n[0,{e}_2] + \ell[0,{e}_1+{e}_2].
    \] 
    For any maximal decomposition $Q \subseteq P$, we have $P = Q$. 
\end{lemma}
\begin{proof}
    First, we claim that it is sufficient to show that the theorem statement holds when $Q$ is a smallest maximal decomposition.
    Given any maximal decomposition $R \subseteq P$, there exists a smallest maximal decomposition contained in $R$, call it $Q'$.
    But, by assumption, the theorem statement holds for $Q'$ so $P = Q' \subseteq R \subseteq P$ which implies $R=P$.
    Thus, we only need to show that the theorem statement holds for smallest maximal decompositions in $P$.

    However, note that if no two of $m,n$ and $\ell$ are $0$, then $\#\partial P = 2L(P) = 2L(Q) = \#\partial Q$.  Hence, by Lemma \ref{lemma:boundary-pt-ineq}, we have $P=Q$.  If two of $m,n$ and $\ell$ are $0$, then $P$ is just a line segment, hence $P=Q.$ \end{proof}

We next develop a formula for computing the minimum distance of the toric code $C_P$ where 
 \[
   P= m[0,{e}_{1}] + n[0,{e}_{2}] + \ell[0,{e}_{1}+{e}_{2}],
\]
for some integers $m,n, \ell \geq 0$. Note, that we may interchange the $m,n$ and $l$ and still be lattice equivalent to $P$, so we may further assume that $0 \leq n \leq m \leq \ell$.  

 To compute the minimum distance, we need to bound the number of zeros of a polynomial $f\in \polyspace_P$.  By \cite[Prop 5.2]{LSchenk}, we know that for $q$ sufficiently large, namely $q \geq (4I(P) +3)^2$, where $I(P)$ denotes the number of interior points of $P$, the reducible sections with more irreducible components necessarily have more zeros in $(\FF_q^\times)^2$ than sections with fewer irreducible components. Thus, a function $f \in \polyspace_P$ with a maximal number of irreducible factors will attain the maximal number of zeros.

\begin{theorem} \label{thrm:bound-maxl-max-zeros}

    Let $P \subseteq [0,q-1]^2$ be the lattice polytope  \[
        P = m[0, {e}_1] + n[0, {e}_2] + \ell[0, {e}_1 + {e}_2]
    \]
    where $0 \leq n \leq m \leq \ell$, 
    then, \[
        \max_{0 \neq f \in \polyspace_P} |Z(f)| \leq L(P)(q-1) - n\ell - m\ell    \]
    when $q \geq (4I(P)+3)^2$.
    
\end{theorem}
\begin{proof}  Suppose that  $0 \leq n \leq m \leq l$ and $q \geq (4I(P)+3)^2$, and let $0 \neq f \in \polyspace_P$ be a polynomial with the maximal number of irreducible factors, i.e $f = f_1 \ldots f_{L(P)}.$ By \cite[Prop 5.2]{LSchenk}, $f$ attains the maximal number of zeros, so it suffices to count the zeros of $f = f_1 \cdots f_{L(P)}.$  Note that the Newton polytope corresponding to $f$, $P_f$, is a maximal decomposition in $P$, and by Lemma \ref{lemma:max-decomps-in-smallest-maxl}, the only maximal decomposition in $P$ is $P$ itself so \[
        f = \underbrace{\prod_{i=1}^m (x-a_i)}_{f_1} \;\; \underbrace{\prod_{i=1}^n (y-b_i)}_{f_2} \;\; \underbrace{\prod_{i=1}^\ell (xy-c_i)}_{f_3} 
    \]
    for some $a_i, b_i, c_i \in \FF_q^\times$.

    Letting $A,B,C$ denote the number of distinct $a_i,b_i,c_i$, respectively, then \begin{align*}
        |Z(f)| &= |Z(f_1)| + |Z(f_2)| + |Z(f_3)| \\
        &\hspace{5mm} - \left( |Z(f_1) \cap Z(f_2)| + |Z(f_1) \cap Z(f_3)| + |Z(f_2) \cap Z(f_3)| \right) \\
        &\hspace{5mm} + |Z(f_1) \cap Z(f_2) \cap Z(f_3)| \\
        &= (A+B+C)(q-1) - (AB + AC + BC) + \#\{i,j,k : a_ib_j = c_k\}.
    \end{align*}
    
    If $A=m$, $B=n$ and $C= \ell$ (i.e the  $a_i$ are distinct, $b_i$ are all distinct and $c_i$ are all distinct) then
     \begin{equation*} 
        |Z(f)| = L(P)(q-1) - m\ell - n\ell - mn + \#(i,j,k : a_ib_j = c_k) \leq L(P)(q-1) - m\ell - n\ell - mn + \min \left\{ m\ell, n\ell, mn \right\} 
    \end{equation*}
    where the final inequality follows from the fact that $\#\{i,j,k : a_ib_j = c_k\}$ is bounded above by $\ell m, \ell n, mn$.
     Thus we have
\begin{equation*} \label{eq:desired-eq}
        |Z(f)| \leq L(P)(q-1) - m\ell - n\ell. 
    \end{equation*}

    Now suppose that $A+B+ C < m+n+\ell$, then $ |Z(f)|  \leq L(P)(q-1) - m\ell - n\ell $, since
 \begin{align*}
        L(P)(q-1) - \ell(m+n) - |Z(f)| &= (m+n+\ell-A-B-C)(q-1) \\
        &\hspace{5mm} - (m\ell + n\ell - AC - BC) + \left( AB - \#\{i,j,k : a_ib_j=c_k\} \right) \\
        &\ge (m+n+\ell-A-B-C)(mn+m\ell+n\ell)   && \text{ since $q > mn+m\ell+n\ell$ } \\
        &\hspace{5mm} - (m\ell + n\ell - AC - BC) - \#\{i,j,k : a_ib_j=c_k\}    && \text{ and $AB \ge 0$} \\
        &\ge (mn+m\ell+n\ell) - (m\ell + n\ell - AC - BC)\\  
        &\hspace{5mm} - \#\{i,j,k : a_ib_j=c_k\} \\
        &\ge mn - \#\{i,j,k : a_ib_j=c_k\} \\
        &\ge 0
    \end{align*}
    where the last inequality follows from the bound $\#\{i,j,k : a_ib_j=c_k\} \le mn$.    
    
   \end{proof}

   \begin{theorem} \label{thrm:bound-maxl-max-zeros} 
    Let $P \subseteq [0,q-1]^2$ be the lattice polytope  
    \[
        P = m[0, {e}_1] + n[0, {e}_2] + \ell[0, {e}_1 + {e}_2]
    \]
    where $0 \leq n \leq m \leq \ell$,  and suppose there exists an integer $t$, with $m \leq t \leq \ell$ and $t \mid q-1$
    then, \[
        \max_{0 \neq f \in \polyspace_P} |Z(f)| = L(P)(q-1) - n\ell - m\ell    \]
    when $q \geq (4I(P)+3)^2$. 
\end{theorem}
 
 \begin{proof}
 By the previous theorem we have that the number of zeros of any polynomial $ f \in \polyspace_P$ is bounded above by $L(P)(q-1) - n\ell - m\ell$.  We will now construct a specific polynomial that reaches the bound, assuming that there is an integer $t$ with $m \leq t \leq \ell$ and $t| q-1$.  Recall that the set of nonzero elements of $\FF_q$ with the operation of multiplication form a cyclic group of order $q-1$, so there is an element $\alpha \in \FF_q^\times$ of order $t$, i.e. $\langle \alpha \rangle = \{ \alpha, \alpha^2, \dots, \alpha^t=1\}$ is a subgroup of $\FF_q^\times.$  Choose $\ell - t$ distinct elements $c_{t+1}, \dots, c_\ell \in \FF_q^\times \setminus  \{ \alpha, \alpha^2, \dots, \alpha^t=1\}$, and let $\mathcal{C} = \{ \alpha, \alpha^2, \dots, \alpha^t, c_{t+1}, \dots, c_\ell\}$, $\mathcal{B} = \{ \alpha, \alpha^2, \dots, \alpha^n\}$,  $\mathcal{A} = \{ \alpha, \alpha^2, \dots, \alpha^m\}$. Define 
 
 \[f = \prod_{a\in \mathcal{A}} (x-a)\prod_{b\in \mathcal{B}} (y-b)\prod_{c\in \mathcal{C}} (xy-c),\]
 then
    \[|Z(f)| =  (m+n+\ell)(q-1) - (mn + m\ell + n\ell) + | Z(\prod_{a \in \mathcal{A}} (x-a)) \cap Z(\prod_{b \in \mathcal{B}} (y-b)) \cap Z(\prod_{c \in \mathcal{C}} (xy-c)) | \]
    
  Now, $Z(\prod_{a \in \mathcal{A}} (x-a)) \cap Z(\prod_{b \in \mathcal{B}} (y-b)) = \mathcal{A} \times \mathcal{B}$, and note that if $(a,b) \in \mathcal{A} \times \mathcal{B}$, say $a= \alpha^i$ for some $1\leq i \leq m$ and $b= \alpha^j$ for some $1 \leq j \leq n$, then  $ab = \alpha^{i+j} \in \langle \alpha \rangle \subseteq \mathcal{C}$. Thus any element  $(a,b) \in \mathcal{A} \times \mathcal{B}$ is a zero of $xy-ab$, hence a zero of  $\prod_{c \in \mathcal{C}} (xy-c)$.  Therefore, \[Z(\prod_{a \in \mathcal{A}} (x-a)) \cap Z(\prod_{b \in \mathcal{B}} (y-b)) \cap Z(\prod_{c \in \mathcal{C}} (xy-c)) = (\mathcal{A} \times \mathcal{B}) \cap  Z(\prod_{c \in \mathcal{C}} (xy-c)) =\mathcal{A} \times \mathcal{B},\] which has $mn$ elements.  Thus
   \[|Z(f)| =  (m+n+\ell)(q-1) - (mn + m\ell + n\ell) + mn = L(P)(q-1)-n\ell - m \ell.\]

  \end{proof}

\begin{theorem} \label{cor:smallest-maxl-min-dist}
    Let \[
        P = m[0,{e}_{1}] + n[0,{e}_{2}] + \ell[0,{e}_{1}+{e}_{2}] \subseteq [0,q-1]^2,
    \]
    where $0\leq n \leq m \leq \ell$, and suppose that there is an integer $t$ with $m \leq t \leq \ell$ such that $t| q-1$, then 
    \[
        d(C_{P}) = (q-1)^{2} - L(P)(q-1) +
            \ell(m+n)        
    \]
    when $q \geq (4I(P)+3)^2$. 
\end{theorem}

\subsection{Polytopes which are the Minkowski sum of three line segments and a $2$-simplex}
As a second application, we will compute the minimum distance of a toric code given by a lattice polytope $P \subset \RR^2$, where $P$ is the Minkowski sum of line segments and $2$-simplices.  Let $\Delta \subset \RR^2$ be the unit $2$-simplex, and recall that by Proposition \ref{prop:SSMinkResults},  $L(r\Delta)=r$. 
We begin by determining the maximal decompositions in  $r\Delta \subseteq \RR^2.$

\begin{lemma} \label{lemma:maxl-decomp-in-simplex}
    Let $r > 0$ and $Q$ a  maximal decomposition in $r\Delta \subseteq \RR^2$, then  \[
    Q = ce_2 + a[0,e_{1}] + b[0,e_{2}] + c[0,e_{1} - e_{2}] + d\Delta
    \]
    for some integers $ a,b,c,d \geq 0$ satisfying $a+b+c+d=r$.  
\end{lemma}
\begin{proof}

    We begin by determining that all maximal decompositions in $r\Delta$ which are zonotopes must take the form in the statement with $d=0$.
   Let $Z$ be a maximal decomposition in $r\Delta$ which is a zonotope.
    Then, by \cite[Proposition 3.1]{SS1}, there can be at most four distinct summands in $Z$.
    Therefore, $Z$ must take the form: \[
        Z = y + \alpha_1[0,v_1] + \alpha_2[0,v_2] + \alpha_3[0,v_3] + \alpha_4[0,v_4]
    \]
    where $y \in Z$, $\sum_i \alpha_i = r$, and each $v_i \in \ZZ^2$ primitive.

We first prove that the only possibilites for $v_i$ are $e_1 = (1,0), e_2=(0,1)$ and $e_1-e_2 = (1,-1)$.  To do so, suppose first that $v_i=(a_i,b_i)$ is primitive with both $a_i,b_i\geq0$, and further assume that one coordinate is greater than or equal to $2$.  Then we can write $Z= a + [0, v_i] + Z',$ where $Z' \subseteq(r-2)\Delta.$  Thus $L(Z') \leq L((r-2)\Delta) = r-2,$ and $L(Z) \leq L([0,v_i]) + L(Z') \leq 1+r-2 = r-1.$  But this is a contradiction, as $Z$ is a maximal decomposition in $r\Delta$, and hence $L(Z)=L(r\Delta)=r$.  So we conclude that if 
$v_i=(a_i,b_i)$ is primitive with both $a_i,b_i\geq0$, then $v_i =(0,1), (1,0)$ or $(1,1).$  However, note that if $v_i=(1,1),$ then we can again write $Z=a + [0, e_1+e_2] + Z',$ where $Z' \subseteq (r-2)\Delta,$ and arrive at a contradiction.  Similarly, we cannot have both coordinates of $v_i$ negative, and the only possibility when the coordinates have opposite sign is $e_1-e_2.$   Thus, if $Z\subseteq r\Delta$ is a maximal decompostion and a zonotope, 
 \[
        Z = y + \alpha_1[0,e_1] + \alpha_2[0,e_2] + \alpha_3[0,e_1-e_2] 
    \]
    where $\sum_i \alpha_i = r$.

    Now let $Q \subseteq r\Delta$ be any maximal decomposition.  By Proposition \ref{prop:SS1MinkLength}(d), any maximal decompostion in $P$ contains a zonotope maximal in $P$, so the only line segments that can appear as a summand in $Q$ are those with direction vectors $e_1, e_2$ and $e_1-e_2.$  There are only two ways to make a two-dimensional Minkowski length $1$ polytope out of these, namely $\Delta$ and $-\Delta.$   If $Q=-\Delta + Q' \subseteq r\Delta$, then we again conclude that $Q' \subseteq (r-2)\Delta,$ which contradicts the fact that $L(Q)=r.$ Thus 
    \[Q = ce_2+ a[0,e_{1}] + b[0,e_{2}] + c[0,e_{1} - e_{2}] + d\Delta
    \]
    for some integers $ a,b,c,d \geq 0$ satisfying $a+b+c+d=r$.

\end{proof}

\begin{lemma}\label{lem:special-quad-clipped-rect-max-decomp}
Let $P \subset \RR^{2}$ be the lattice polytope 
\[
        P = m[0,{e}_{1}] + n[0,{e}_{2}] + \ell[0,{e}_{1}+{e}_{2}] + s[0,{e}_{1}-{e}_{2}] + 2\ell\Delta
    \]
    where $m,n,s,\ell \ge 0$.  Then
    \[L(P)=m+n+s+4l,\]
    and if $Q$ is a maximal decomposition in $P$, then \[Q \approx (m+2\ell)[0,{e}_{1}] + (n+2\ell)[0,{e}_{2}] + s[0,{e}_{1}+{e}_{2}].
    \]
    
\end{lemma}
\begin{proof}
     Let  $P = m[0,{e}_{1}] + n[0,{e}_{2}] + \ell[0,{e}_{1}+{e}_{2}] + s[0,{e}_{1}-{e}_{2}] + 2\ell\Delta$ (see the black polytope in Figure \ref{fig:special-quad-clipped-rect-max-zeros}).
     
   \begin{figure}[h]
        \includegraphics[scale=0.3]{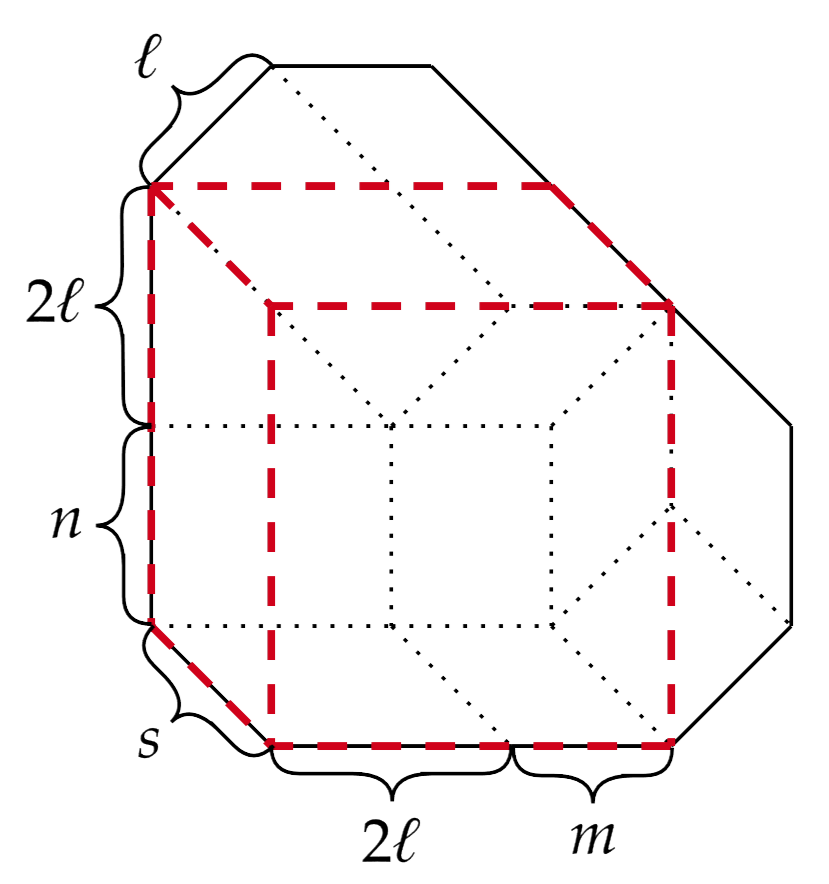}
        \caption{$P$ (black) and $Q$ (red) from Lemma \ref{lem:special-quad-clipped-rect-max-decomp}}
        \label{fig:special-quad-clipped-rect-max-zeros}
    \end{figure}
     
    Note that by Proposition \ref{prop:quad-clipped-rect-minkowski}, $L(P) = m+n+s+4\ell$, and $P ':= s{e}_{2} + P \subset (m+n+s+4\ell)\Delta$ so if $Q$ is a maximal decomposition in $P'$ then $Q$ is a maximal decomposition in $(m+n+s+4\ell)\Delta$.
    By Lemma \ref{lemma:maxl-decomp-in-simplex}, \[
        Q = ce_2+ a[0,{e}_{1}] + b[0,{e}_{2}] + c[0,{e}_{1} - {e}_{2}] + d\Delta \subseteq  P'
    \]
    where  $a,b,c,d \ge 0$ and $a+b+c+d = m+n+s+4\ell$.
    Note that $Q$ is not contained in $P'$ if $c < s$, so we must have $c \ge s$.
    Furthermore, $c{e}_{2} + b{e}_{2} + d{e}_{2} \in Q$ so $b+c+d \le n+s+2\ell$ and $c{e}_{1} + a{e}_{1} + d{e}_{1} \in Q$ so $a+c+d \le m+s+2\ell$ which implies 
    \[
        m+n+2s+4\ell \geq a+b+2c+2d = (m+n+s+4\ell)+c+d.
    \]

    Thus, we conclude that $s \geq c+d$.
    As $c \ge s$, we must have $d=0$ and $c=s$.
    This further implies $b \le n+2\ell$ and $a \le m+2\ell$, so $b=n+2\ell$ and $a=m+2\ell$ as $a+b+c < m+n+s+4\ell$ otherwise.
    Therefore, \[Q=
        se_2 + (m+2\ell)[0,{e}_{1}] + (n+2\ell)[0,{e}_{2}] + s[0,{e}_{1}-{e}_{2}] \approx (m+2\ell)[0,{e}_{1}] + (n+2\ell)[0,{e}_{2}] + s[0,{e}_{1}+{e}_{2}].
    \]
    Note that we have shown that $Q$ is the only maximal decomposition in $P$.

\end{proof}

\begin{theorem} \label{cor:special-quad-clipped-rect-min-dist}  
    Let \[
        P = m[0,{e}_{1}] + n[0,{e}_{2}] + \ell[0,{e}_{1}+{e}_{2}] + s[0,{e}_{1}-{e}_{2}] + 2\ell\Delta \subseteq [0,q-1]^2
    \]
    where
 $0\leq n+ 2 \ell \leq m + 2\ell \leq s$, and there is an integer $t$ with $m + 2 \ell \leq t \leq s$ such that $t| q-1$.
    Then, \[
        d(C_{P}) = (q-1)^{2} - L(P)(q-1) + s(m+n+4\ell)
    \]
    when $q \geq (4I(P)+3)^2$. 
     
\end{theorem}

\begin{proof}
    By Lemma \ref{lem:special-quad-clipped-rect-max-decomp},  any maximal decomposition $Q$ in $P$ is of the form $Q= (m+2\ell)[0,{e}_{1}] + (n+2\ell)[0,{e}_{2}] + s[0,{e}_{1}+{e}_{2}].$  Applying Theorem \ref{thrm:bound-maxl-max-zeros}, we have that when $q \geq (4I(P)+3)^2$,
    \[
        \max_{0 \neq f \in \polyspace_Q} |Z(f)| = L(Q)(q-1) - (n+2\ell)s - (m+2\ell)s    = L(Q)(q-1) - s(m+n+4\ell).   \]
    Thus,
    \[
        d(C_{P}) = (q-1)^{2} - L(P)(q-1) + s(m+n+4\ell).
    \]

\end{proof}

While this method offers a new approach to computing minimum distance, all the applications of this method showcased are quite simple.
Namely, both examples had only one maximal decomposition and both of those maximal decompositions only contained (at most) three distinct summands.
While the fact that both examples only contained one maximal decomposition made the computation easier, a formula could still be easily derived if there were multiple maximal decompositions (all of which contained at most three distinct summands) by simply taking the maximum over all maximal decompositions.
However, the fact that both maximal decompositions contained at most three distinct summands made deriving a sharp bound on the maximum number of $\FF_{q}^{\times}$-zeros for polynomials whose Newton polytopes correspond to these maximal decompositions using an inclusion-exclusion argument much easier.
When a maximal decomposition contains four distinct summands, deriving a sharp upper bound on $\FF_{q}^{\times}$-zeros using an inclusion-exclusion argument becomes much more difficult.

For example, consider the ``clipped rectangle'' (see Figure \ref{fig:clipped-rect}) \[
    R = m[0,{e}_{1}] + n[0,{e}_{2}] + p\Delta \subseteq (m+n+p)\Delta, \, \text{ where $m,n,p \ge 0$. }
\]
 
\begin{figure}[h]
    \centering
    \begin{tikzpicture}[scale=1.3]
        \draw[line width=.4mm] (0,0) -- (2,0) -- (3,0) -- (3,2) -- (2,3) -- (0,3) -- (0,2) -- (0,0);
        \draw[dashed, line width=.3mm] (0,2) -- (2,2) -- (2,0);
        \draw[dashed, line width=.3mm] (2,2) -- (2,3);
        \draw[dashed, line width=.3mm] (2,2) -- (3,2);

        \draw[decorate, decoration={calligraphic brace,raise=1mm}, line width=.5mm] (2,0) -- node[below=1.5mm] {\large $m$} (0,0);
        \draw[decorate, decoration={calligraphic brace,raise=1mm}, line width=.5mm] (0,0) -- node[left=1.5mm] {\large $n$} (0,2);
        \draw[decorate, decoration={calligraphic brace,raise=1mm}, line width=.5mm] (0,2) -- node[left=1.5mm] {\large $p$} (0,3);
    \end{tikzpicture}
    \caption{The ``clipped rectangle''}
    \label{fig:clipped-rect}
\end{figure}

As $L(R)=m+n+p$ by Proposition \ref{prop:quad-clipped-rect-minkowski} and $R \subseteq (m+n+p)\Delta$, if $Q$ is a maximal decomposition in $R$ then \[
    Q =y + a[0,{e}_{1}] + b[0,{e}_{1}] + c[0,{e}_{1}-{e}_{2}] + d\Delta 
\]
where $a,b,c,d \ge 0$ and $a+b+c+d = m+n+p$ by Lemma \ref{lemma:maxl-decomp-in-simplex}.
We may further restrict $Q$ to those so that $a \le m+p$, $b \le n+p$, and $c+d \le p$.
However, this still leaves $Q$ with four distinct summands unlike our previous examples.

We will address a specific class of these ``clipped rectangles,'' namely the case when $m=n=p$, using the method of Vandermonde matrices developed by Little and Schwarz \cite{LSchwarz}.

\section{Using Vandermonde Matrices to Compute Minimum Distance} \label{sec:ls-method}

Consider the lattice polytope $P=\ell\Delta+\ell[0,e_1] +\ell[0,e_2]\subseteq [0,q-2]^2$, which, by Proposition \ref{prop:quad-clipped-rect-minkowski}, has period $1$, and its associated code $C_P.$ 
We will use the methods of \cite{LSchwarz} to derive an explicit formula for the minimum distance of $C_P$.
Note that polytopes of this form have exactly $(2\ell+1)^2-\frac{\ell(\ell+1)}{2}$ lattice points, which we will denote as $\#P$.

Let $R_{2\ell \times \ell}:= 2\ell[0,e_1] \times \ell[0,e_2]$ denote the $2 \ell \times \ell$ rectangle, and note that $R_{2\ell \times \ell} \subset \ell\Delta+\ell[0,e_1] +\ell[0,e_2]$.   Thus the minimum distance of $C_P$ is bounded above by the minimum distance of $C_{R_{2\ell\times\ell}}$. 
By \cite[Theorem 3]{LSchwarz}, $d(C_{R_{2\ell\times\ell}})=(q-1-2\ell)(q-1-\ell)$. Thus, we have the following upper bound:  \[d(C_P)\leq (q-1-2\ell)(q-1-\ell)\]

To bound the minimium distance from below, we will use the method of Vandermonde matrices, introduced in \cite{LSchwarz}.

\begin{defn}
    Let $P \subset \RR^m$ be a lattice polytope with the lattice points of $P$, $P\cap\ZZ^m=\{e(i):1\leq i \leq\#P\}$ listed in a particular order. Let $S=\{p_j:1\leq j \leq \#P\}$ be any ordered set of $\#P$ points in $(\FF_q^\times)^m$. Then the \emph{Vandermonde matrix}, $V(P;S)$, corresponding to $P$ and $S$ is the $\#P\times\#P$ matrix \[V(P;S)=\left(p_j^{e(i)}\right)\] where each $p_j^{e(i)}$ represents the value of the monomial $x^{e(i)}$ evaluated at the point $p_j$. 
    
\end{defn}

\begin{theorem}\cite[Proposition 1]{LSchwarz} \label{thm:LSProp1}
Let $P \subset \RR^m$ be a lattice polytope. Let $d$ be a positive integer and assume that in every set $T \subseteq (\FF^\times_
q )^m$ with $|T | = (q - 1)^m - d + 1$ there exists
some $S \subseteq T$ with $|S| = \#P$ such that $\det V (P; S) \neq 0$. Then the minimum distance satisfies $d(C_P ) \geq d$.
\end{theorem}

Let $P=\ell\Delta+\ell[0,e_1] +\ell[0,e_2]\subseteq [0,q-2]^2$. We will show that for every set $T\subset(\FF_q^\times)^2$ with $|T|=(q-1)^2-d+1$, where $d=(q-1-2\ell)(q-1-\ell)$, there exists some $S\subset T$ with $|S|=\#P$ such that the determinant of the Vandermonde matrix corresponding to $P$ and our set $S$ is nonzero, giving us the lower bound $d(C_P)\geq d$. 

We begin by defining a particularly nice configuration of points.  

\begin{defn}
    Let $\ell$ be a positive integer. A collection $S$ of $(2\ell+1)^2-\frac{\ell(\ell+1)}{2}$ points in $(\mathbb{F}_q^\times)^2$ is a \textbf{length} $\ell$ \textbf{staircase} configuration if the following conditions hold:
\begin{enumerate}
    \item There exist vertical lines $x=a_1,x=a_2,\ldots, x=a_{\ell+1}$ such that $S$ intersects each vertical line in exactly $2\ell+1$ lattice points.
    \item For each integer $j$, $0\leq j \leq \ell-1$ there exists a vertical line $x=b_j$ such that $S$ intersects $x=b_j$ in exactly $2\ell-j$ lattice points.
\end{enumerate}
\end{defn}

\begin{example} 
  Below are staircase configurations of length $\ell =1$ and length $\ell=2$
\begin{figure}[h]
\begin{tikzpicture}[scale=.8, thick, dot/.style={circle, fill=black, inner sep=2pt}]

    \node[dot]  at (0,0) {};
    \node[dot]  at (0,1) {};
    \node[dot]  at (0,2) {};
    \node[dot]  at (1,0) {};
    \node[dot]  at (1,1) {};
    \node[dot]  at (1,2) {};
    \node[dot]  at (2,0) {};
   \node[dot]  at (2,1) {};
    \draw[anchor=north] node at (1,-.3) { $\ell =1$};

    \node[dot]  at (5,0) {};
    \node[dot]  at (5,1) {};
    \node[dot]  at (5,2) {};
    \node[dot]  at (5,3) {};
    \node[dot]  at (5,4) {};
    \node[dot]  at (6,0) {};
    \node[dot]  at (6,1) {};
    \node[dot]  at (6,2) {};
    \node[dot]  at (6,3) {};
    \node[dot]  at (6,4) {};
    \node[dot]  at (7,0) {};
    \node[dot]  at (7,1) {};
    \node[dot]  at (7,2) {};
    \node[dot]  at (7,3) {};
    \node[dot]  at (7,4) {};
    \node[dot]  at (8,0) {};
    \node[dot]  at (8,1) {};
    \node[dot]  at (8,2) {};
    \node[dot]  at (8,3) {};
    \node[dot]  at (9,0) {};
    \node[dot]  at (9,1) {};
    \node[dot]  at (,2) {};
    \draw[anchor=north] node at (6,-.3) { $\ell =2$};

\end{tikzpicture}
\caption{Staircase configurations of length $\ell=1, 2$}
\label{fig:scale}
   \end{figure}
\end{example}

\begin{remark}
    If a set $S\subset (\FF_q^\times)^2$ satisfies the requirements to be a staircase configuration, then $|S|=(\ell+1)(2\ell+1)+\sum_{j=0}^{\ell-1}(2\ell-j) = \#P$.
\end{remark}

We will prove in Proposition \ref{prop:non-zero-det} that the determinant of the multivariate Vandermonde matrix corresponding to polytopes of the form $P=\ell\Delta+\ell[0,e_1] +\ell[0,e_2]$, and an appropriate length $\ell$ staircase configuration is always non-zero. Thus, in order to apply Theorem \ref{thm:LSProp1} we want to show that for every $T\subset(\FF_q^\times)^2$ with $|T|=(q-1)^2-d+1$, with $d=(q-1-2\ell)(q-1-\ell)$, that there exists a length $\ell$ staircase configuration, $S\subset T$.

\begin{lemma}\label{lem:SinT}
     Suppose that $q>2\ell +1$ and let $d=(q-1-2\ell)(q-1-\ell)$.  Then every subset $T\subset(\mathbb{F}_q^\times)^2$ with $|T|=(q-1)^2-d+1$ contains a staircase configuration of length $\ell.$

\end{lemma}

\begin{proof}
   
 $|T|>2\ell(q-1)$ rewrites as $\ell q+1>2\ell^2+\ell$, which holds true since $q$ is assumed to be greater than $2\ell+1$. To show that $T$ contains $\ell+1$ lines, $x=a_1,\ldots,x=a_{\ell+1}$, with at least $2\ell+1$ points on each line, we must show that: \[(3\ell q - 2\ell^2-3\ell+1)-i(q-1)>2\ell(q-1-i),\:\:\text{for } i\in \mathbb{Z}^+, 0\leq i \leq \ell.\]
This is equivalent to showing that \[\ell q-2\ell^2-\ell+1>i(q-1-2\ell),\:\: \text{for  } i\in \mathbb{Z}^+, 0\leq i \leq \ell.\]
Since $q>2\ell+1$, and $i\leq \ell$, it suffices to check this for $i=\ell$, which is immediate.

Thus, by the pigeonhole principle, we know that for all $q>2\ell+1$, each subset $T\subset(\mathbb{F}_q^\times)^2$ must contain $\ell+1$ lines $x=a_1,\ldots,x=a_{\ell+1}$ each containing at least $2\ell+1$ lattice points.

To show that for each integer $j$, $0\leq j \leq \ell-1$ there exists a line $x=b_j$ such that each line contains $2\ell-j$ points, we must show that $ |T| - (\ell+1+j)(q-1)>(2\ell-1-j)(q-2-\ell-j)$. This simplifies to $\ell(j+1)>(j+2)j$, which is obvious for $\ell\geq j+2$, and also holds when $\ell=j+1$.

\end{proof}

\begin{proposition}\label{prop:non-zero-det}
    Let $P$ be a lattice polytope in $\RR^2$ such that $P \approx \ell\Delta+ \ell [0,e_1] +\ell[0,e_2]$ and let $S$ be a length $\ell$ staircase configuration. Then the Vandermonde matrix $V(P;S)$ corresponding to $P$ and our set $S$ has non-zero determinant. 
\end{proposition}

\begin{proof}
    The Vandermonde matrix has entries made up of monomials corresponding to points in $P$, so the determinant will be a linear combination of monomials from $P$.  We want to to show that when we evaluate the determinant at $S$, it is non-zero, so it suffices to show that any any linear combination $f(x,y)$ of monomials from $P$ which vanishes at the points in $S$ is the zero polynomial. 
    
    To this end, take any linear combination of monomials from $P$:
    \[f(x,y) = g_1(x) + yg_2(x) + \cdots + y^{2\ell}g_{2\ell+1}(x),\]
    where $g_1, \dots, g_{2\ell+1}$ are polynomials in $x$. Note $\deg(g_i) \leq 2\ell$ for $1 \leq i \leq \ell+1$ and $\deg(g_i) \leq 3\ell -i +1$ for $\ell+2 \leq i \leq 2\ell + 1$. Assume $f$ vanishes at all points in $S$.


    Next, consider any value of $a_i$ with $1 \leq i \leq \ell+1$ from the definition of the staircase configuration. This means $f(a_i,y)$ is a polynomial in $y$ of degree $2\ell$ that vanishes at the $2\ell+1$ points in $S$ along the line $x=a_i$. Thus $f(a_i,y) \equiv 0$, so $g_k(a_i) = 0$ for all indices $k$ and all $a_i$ with $1 \leq i \leq \ell+1$. In particular, $g_{2\ell + 1}$ has at most degree $\ell$ and disappears at $\ell + 1$ points (all $a_i$), meaning $g_{2\ell + 1} \equiv 0$. 

    Therefore, $f$ in fact has the form
    \[f(x,y) = g_1(x) + yg_2(x) + \cdots +y^{\ell+1}g_{\ell+2}(x) + y^{\ell+2}g_{\ell+3}(x) + \cdots + y^{2\ell-1}g_{2\ell}(x).\]

    Now, $f(b_0,y)$ has degree in $y$ of at most $2\ell -1$ and vanishes at the $2\ell$ points in $S$ along $x=b_0$, so is identically zero. Combining this with the previous part, all $g_k$ vanish at all $a_i$ and at $b_0$, a total of $\ell+2$ points. As $g_{2\ell}$ has degree at most $\ell +1$, it must vanish entirely, meaning $f(x,y)$ actually has degree of at most $2\ell -2$ in $y$. Call this step $0$ and iterate the process. 

    At step $n$, we begin with  $f(b_{n},y)$ having degree of at most $2\ell - n-1$ and vanishing at $2\ell - n $ points along $x=b_{n}$. Immediately, $g_k(b_n)=0$ for all indices $k$. Then, $g_{2\ell - n}$ has degree of at most $\ell+n+1$ and vanishes at $\ell + n+2$ points (all $a_i$ and an additional $b_j$ at each step), so is identically zero. This reduces the degree of $f$ in $y$ by one, and the process repeats. 
    
    Once the process terminates at $n=\ell-1$, it has shown that all $g_k$ vanish at $2\ell + 1$ points. As $\deg g_k \leq 2\ell$ for all $k$, all $g_k \equiv 0$, which concludes the proof.   
\end{proof}

\begin{theorem}\label{cor:Van-Min-Dist}
    Let $P \approx \ell\Delta+ \ell [0,e_1] +\ell[0,e_2]$, for some positive integer $\ell$, with $P \subseteq [0,q-1]^2.$ Then the toric code $C_P$ has minimum distance
    \[d(C_P)= (q-1-2\ell)(q-1-\ell)\]
\end{theorem}
\begin{proof}
 Let $d=(q-1-\ell)(q-1-\ell)$.  Since $P$ contains the $2\ell \times \ell$ rectangle, we have that 
\[d(C_P) \leq d(R_{2\ell \times \ell}) = (q-1-2 \ell)(q-1-\ell)=d.\]
By Lemma \ref{lem:SinT}, every subset $T \subset (\FF_q^\times)^2$ with $|T|=(q-1)^2-d+1$ contains a staircase configuration $S$ of length $\ell$, where $|S|= \#P$.  By Proposition \ref{prop:non-zero-det} the Vandermonde matrix $V(P;S)$ has non-zero determinant, so by Theorem \ref{thm:LSProp1}, $d(C_P) \geq d.$

\end{proof}

\section*{Acknowledgments}

This research was completed at the summer 2024 REU Site: Mathematical Analysis and Applications at the University of Michigan-Dearborn. We would like to thank the National Science Foundation (DMS-1950102 and DMS-2243808), the National Security Agency (H98230-24), the College of Arts, Sciences, and Letters, and the Department of Mathematics and Statistics for their support. Additionally, we thank the other participants of the REU program for fruitful conversations on this topic.  We further thank the referee for helping simplify many of the proofs and offering suggestions to make the paper more accessible to undergraduate readers.

\bibliography{REUbib}

@article{Beckwith_2012,
	author = {Beckwith, Olivia and Grimm, Matthew and Soprunova, Jenya and Weaver, Bradley},
	doi = {10.1007/s00454-012-9433-5},
	issn = {1432-0444},
	journal = {Discrete \& Computational Geometry},
	month = {Jun},
	publisher = {Springer Science and Business Media LLC},
	title = {Minkowski Length of 3D Lattice Polytopes},
	url = {http://dx.doi.org/10.1007/s00454-012-9433-5},
	year = {2012}}

@article{SS1,
	author = {Soprunov, Ivan and Soprunova, Jenya},
	doi = {10.1137/080716554},
	fjournal = {SIAM Journal on Discrete Mathematics},
	issn = {0895-4801},
	journal = {SIAM J. Discrete Math.},
	mrclass = {94B27 (14G50 52B20)},
	mrnumber = {2476837},
	mrreviewer = {Matthias Beck},
	number = {1},
	pages = {384--400},
	title = {Toric surface codes and {M}inkowski length of polygons},
	url = {https://doi.org/10.1137/080716554},
	volume = {23},
	year = {2008/09}}

@article{SS2,
	author = {Soprunov, Ivan and Soprunova, Jenya},
	doi = {10.1137/090762592},
	fjournal = {SIAM Journal on Discrete Mathematics},
	issn = {0895-4801},
	journal = {SIAM J. Discrete Math.},
	mrclass = {94B27 (14G50 52B20)},
	mrnumber = {2661429},
	mrreviewer = {Seon Jeong Kim},
	number = {2},
	pages = {655--665},
	title = {Bringing toric codes to the next dimension},
	url = {https://0-doi-org.wizard.umd.umich.edu/10.1137/090762592},
	volume = {24},
	year = {2010}}

@article{Joyner04,
	author = {Joyner, David},
	doi = {10.1007/s00200-004-0152-x},
	fjournal = {Applicable Algebra in Engineering, Communication and Computing},
	issn = {0938-1279},
	journal = {Appl. Algebra Engrg. Comm. Comput.},
	mrclass = {94B25 (94B27)},
	mrnumber = {2142431},
	mrreviewer = {Sapna Jain},
	number = {1},
	pages = {63--79},
	title = {Toric codes over finite fields},
	url = {https://doi.org/10.1007/s00200-004-0152-x},
	volume = {15},
	year = {2004}}

@article{LSchwarz,
	author = {Little, John and Schwarz, Ryan},
	doi = {10.1007/s00200-007-0041-1},
	fjournal = {Applicable Algebra in Engineering, Communication and Computing},
	issn = {0938-1279},
	journal = {Appl. Algebra Engrg. Comm. Comput.},
	mrclass = {94B27 (52B20)},
	mrnumber = {2322944},
	mrreviewer = {Jos\'{e} Ignacio Farr\'{a}n},
	number = {4},
	pages = {349--367},
	title = {On toric codes and multivariate {V}andermonde matrices},
	url = {https://doi.org/10.1007/s00200-007-0041-1},
	volume = {18},
	year = {2007}}

@article{LSchenk,
	author = {Little, John and Schenck, Hal},
	doi = {10.1137/050637054},
	fjournal = {SIAM Journal on Discrete Mathematics},
	issn = {0895-4801},
	journal = {SIAM J. Discrete Math.},
	mrclass = {14G50 (94B27)},
	mrnumber = {2272243},
	number = {4},
	pages = {999--1014},
	title = {Toric surface codes and {M}inkowski sums},
	url = {https://doi.org/10.1137/050637054},
	volume = {20},
	year = {2006}}

@article{Hansen,
	author = {Hansen, Johan P.},
	doi = {10.1007/s00200-002-0106-0},
	fjournal = {Applicable Algebra in Engineering, Communication and Computing},
	issn = {0938-1279},
	journal = {Appl. Algebra Engrg. Comm. Comput.},
	mrclass = {14G50 (14M25 94B27)},
	mrnumber = {1953195},
	mrreviewer = {John B. Little},
	number = {4},
	pages = {289--300},
	title = {Toric varieties {H}irzebruch surfaces and error-correcting codes},
	url = {https://doi.org/10.1007/s00200-002-0106-0},
	volume = {13},
	year = {2002}}

@inproceedings{Hansen98,
	address = {Berlin, Heidelberg},
	author = {Hansen, Johan P.},
	booktitle = {Coding Theory, Cryptography and Related Areas},
	editor = {Buchmann, Johannes and H{\o}holdt, Tom and Stichtenoth, Henning and Tapia-Recillas, Horacio},
	isbn = {978-3-642-57189-3},
	pages = {132--142},
	publisher = {Springer Berlin Heidelberg},
	title = {Toric Surfaces and Error-correcting Codes},
	year = {2000}}

@article{ruano,
	author = {Ruano, Diego},
	doi = {10.1016/j.ffa.2007.02.002},
	fjournal = {Finite Fields and their Applications},
	issn = {1071-5797},
	journal = {Finite Fields Appl.},
	mrclass = {94B27 (14M25 51E22)},
	mrnumber = {2360532},
	mrreviewer = {Trygve Johnsen},
	number = {4},
	pages = {962--976},
	title = {On the parameters of {$r$}-dimensional toric codes},
	url = {https://doi.org/10.1016/j.ffa.2007.02.002},
	volume = {13},
	year = {2007}}

@article {SS3,
    AUTHOR = {Soprunov, Ivan and Soprunova, Jenya},
     TITLE = {Eventual quasi-linearity of the {M}inkowski length},
   JOURNAL = {European J. Combin.},
  FJOURNAL = {European Journal of Combinatorics},
    VOLUME = {58},
      YEAR = {2016},
     PAGES = {107--117},
      ISSN = {0195-6698,1095-9971},
   MRCLASS = {52B05 (11H06 52B20)},
  MRNUMBER = {3530624},
MRREVIEWER = {Matthias\ Schymura},
       DOI = {10.1016/j.ejc.2016.05.009},
       URL = {https://doi.org/10.1016/j.ejc.2016.05.009},
}

\end{document}